\documentclass[12pt]{elsarticle}

\usepackage{amsmath,amssymb,amsthm}

\usepackage{color}

\newtheorem{theorem}{Theorem}[section]
\newtheorem{lemma}{Lemma}[section]

\newtheorem{remark}{Remark}[section]

\theoremstyle{definition}
\newtheorem{definition}{Definition}[section]

\theoremstyle{remark}

\numberwithin{equation}{section}

\usepackage[margin=2.2cm]{geometry}

\newcommand{\R}{\mathbf{R}}
\newcommand{\N}{\mathbf{N}}
\newcommand{\loc}{\textrm{loc}}
\newcommand{\RN}{\mathbf{R}^N}

\def\polhk#1{\setbox0=\hbox{#1}{\ooalign{\hidewidth
    \lower1.5ex\hbox{`}\hidewidth\crcr\unhbox0}}}

\journal{some journal} 

\begin{document}

\begin{frontmatter}



\title{Optimal singularities of initial data for solvability \\
of the Hardy parabolic equation}


\author[adsci]{Kotaro Hisa\corref{cor1}}
\ead{khisa@ms.u-tokyo.ac.jp}

\author[adtitech]{Jin Takahashi}
\ead{takahashi@c.titech.ac.jp}

\address[adsci]{
Graduate School of Mathematical Sciences, The University of Tokyo, \\
3-8-1 Komaba, Meguro-ku, Tokyo 153-8914, Japan 
}

\address[adtitech]{
Department of Mathematical and Computing Science, Tokyo Institute of Technology, \\
2-12-1 Ookayama, Meguro-ku, Tokyo 152-8552, Japan
}

\cortext[cor1]{Corresponding author}

\begin{abstract}
We consider the Cauchy problem for the Hardy parabolic equation 
$\partial_t u-\Delta u=|x|^{-\gamma}u^p$ 
with initial data $u_0$ singular at some point $z$. 
Our main results show that, if $z\neq 0$, then 
the optimal strength of the singularity of $u_0$ at $z$ 
for the solvability of the equation 
is the same as that of the Fujita equation $\partial_t u-\Delta u=u^p$. 
Moreover, if $z=0$, then the optimal singularity 
for the Hardy parabolic equation is weaker than 
that of the Fujita equation. 
We also obtain analogous results for 
a fractional case $\partial_t u+(-\Delta)^{\theta/2} u=|x|^{-\gamma}u^p$ 
with $0<\theta<2$. 
\end{abstract}

\end{frontmatter}



\section{Introduction}
We consider the Cauchy problem for the Hardy parabolic equation 
\begin{equation}\label{eq:main}
\left\{ 
\begin{aligned}
	&\partial_t u-\Delta u=|x|^{-\gamma}u^p
	&&\mbox{ in }\R^N\times(0,T), \\
	&u(\cdot,0)=u_0 &&\mbox{ in }\R^N, 
\end{aligned}
\right.
\end{equation}
where $N\geq1$, $p>1$, $0<\gamma<\min\{2,N\}$ and $0<T\leq \infty$. 
We assume that $u_0$ is nonnegative and has a singularity at some point $z\in \R^N$. 
The aim of this paper is to identify 
the optimal strength of the singularity at $z$ 
for the local-in-time solvability of \eqref{eq:main}. 
Intuitively, such an optimal singularity is determined by 
the diffusion effect and the growth rate of the nonlinear term near $z$. 
Hence, if $z\neq0$, it is expected that the optimal singularity of \eqref{eq:main} 
is the same as that of the Fujita equation $\partial_t u-\Delta u=u^p$. 
On the other hand, if $z=0$, the optimal singularity 
should be weaker than that of the Fujita equation. 
In this paper, we show that these expectations are indeed correct. 
Furthermore, we also give analogous results for 
a fractional case $\partial_t u+(-\Delta)^{\theta/2} u=|x|^{-\gamma}u^p$ 
with $0<\theta<2$, see Section \ref{sec:frac} below.

We recall some results for the problem \eqref{eq:main}. 
In what follows, set
\[
	p_F:= 1+\frac{2}{N}, \qquad 
	p_\gamma:=1+\frac{2-\gamma}{N}. 
\]
Ben Slimene-Tayachi-Weissler \cite{STW} obtained 
sufficient conditions for the solvability of \eqref{eq:main} in $L^q$ spaces. 
They also obtained a global-in-time solution under 
$u_0(x)\leq c |x|^{-(2-\gamma)/(p-1)}$ 
if $p>p_\gamma$ and $c>0$ small. 
For related results, 
see \cite{AT05,S19, N19, NIT, FT00, H08, MM21, P13, Q98, Tayachi, W}. 
Subsequently, the first author and Sier\.{z}\polhk ega \cite{HS} examined 
necessary conditions of initial data for the existence of solutions including 
a fractional case. 
Their results imply that \eqref{eq:main} 
possesses no local-in-time solutions if $u_0$ satisfies either 
\begin{equation}\label{eq:OShd}
	u_0(x)\geq
	\left\{
	\begin{aligned}
	&  C_*|x|^{-N}\left(\log(e+|x|^{-1})\right)^{-\frac{N}{2-\gamma}-1} 
	&& \mbox{ if } \quad \displaystyle{p=p_\gamma}, \\
	& C_*|x|^{-\frac{2-\gamma}{p-1}} 
	&& \mbox{ if } \quad \displaystyle{p>p_\gamma} 
	\end{aligned}
	\right.
\end{equation}
near $x=0$, or 
\begin{equation}\label{eq:OSfjt}
u_0(x)\geq
\left\{ 
\begin{aligned}
	& C_* |z|^\frac{\gamma}{p-1} |x-z|^{-N}
	\left(\log(e+|x-z|^{-1})\right)^{-\frac{N}{2}-1} &&\mbox{ if } p=p_F, \\
	& C_* |z|^\frac{\gamma}{p-1} |x-z|^{-\frac{2}{p-1}} &&\mbox{ if } p>p_F 
\end{aligned}
\right.
\end{equation}
near $x=z$ ($z\neq0$), where 
$C_*>0$ is sufficiently large. 
We note that 
$|x-z|^{-N}(\log(e+|x-z|^{-1}))^{-(N/2)-1}$ and $|x-z|^{-2/(p-1)}$ 
are the same as the optimal singularities 
of the Fujita equation for each case, see \cite{BP85,HI18}. 
In addition, when $u_0=\delta_z$ with $z\in \RN$, 
they \cite{HS} also show that 
the problem \eqref{eq:main} possesses no local-in-time solutions if either 
$p_\gamma\leq p < p_F$ and $z=0$, or $p\geq p_F$ and $z\in\RN$. 
Here $\delta_z$ is the Dirac measure on $\R^N$ concentrated at $z\in\RN$.

The above results in \cite{STW,HS} imply that 
$|x|^{-(2-\gamma)/(p-1)}$ is the optimal singularity of 
the Hardy parabolic equation \eqref{eq:main} in the case where $p>p_\gamma$ and 
the singularity of $u_0$ is located at the origin. 
However, the optimality of the other cases is still open. 
In this paper, we derive sufficient conditions 
for the existence of solutions 
corresponding to \eqref{eq:OShd} and \eqref{eq:OSfjt} 
and specify the optimal singularities of \eqref{eq:main}. 

In order to state our results, we introduce some notation. 
Set 
\begin{equation}\label{eq:psizdef}
	\psi(z):=
	\left\{ 
	\begin{aligned}
	&|z|^\frac{2}{p-1}(1+|z|)^{-\frac{2-\gamma}{p-1}}
	&&\mbox{ if $p<p_F$ with $N\geq1$ or $p=p_F$ with $N\leq 2$}, \\
	&|z|^\frac{\gamma}{p-1}  
	&&\mbox{ if $p=p_F$ with $N>2$ or $p>p_F$ with $N\geq1$}. 
	\end{aligned}
	\right.
\end{equation}
We regard a nonnegative Radon measure $\mu=C\delta_z+\phi$ 
with $\phi\in L^\infty(\R^N)$ as $d\mu(x)=Cd\delta_z(x)+\phi dx$.
We denote by $\chi_1$ the characteristic function on the interval $[0,1]$. 
Each of the solutions in our results is the so-called integral solution, 
see Definition \ref{def:sol} below.  

Our first result shows that if the singularity of $u_0$ is not located at the origin, 
then the optimal singularity of \eqref{eq:main} 
is the same as that of the Fujita equation.

\begin{theorem}\label{th:main}
Fix $N\geq1$, $p>1$, $0<\gamma<\min\{2,N\}$, $0<T<\infty$ and 
$z\in \R^N\setminus\{0\}$. 
Let $\psi$ be as in \eqref{eq:psizdef}. 
Assume either $u_0$ is a nonnegative Radon measure satisfying 
\[
	u_0=c \psi(z) \delta_z + \phi\qquad \mbox{ if }p<p_F 
\]
for a nonnegative function $\phi\in L^\infty(\R^N)$ 
with $\|\phi\|_{L^\infty(\R^N)}\leq C_0$, 
or $u_0$ is a nonnegative measurable function satisfying
\[
	u_0(x)\leq \left\{ 
	\begin{aligned}
	& c \psi(z) |x-z|^{-N} 
	\left( \log ( e+|x-z|^{-1} ) \right)^{-\frac{N}{2}-1}
	\chi_1(|x-z|) + C_0 
	&&\mbox{ if }p=p_F, \\
	& c \psi(z) |x-z|^{-\frac{2}{p-1}} + C_0 
	&&\mbox{ if }p>p_F \\
	\end{aligned}
	\right.
\]
for any $x\in \R^N\setminus\{z\}$. 
Here $c>0$ and $C_0\geq0$ are constants. 
Then there exist positive constants $c_*$ and $C_*$ 
depending on $N$, $p$ and $\gamma$ but not on $T$ and $z$ 
such that the following statements hold. 
If the constants $c$ and $C_0$ satisfy 
\[
	c\leq 
	\left\{
	\begin{aligned} 
	&c_* T^{-\frac{N(p_F-p)}{2(p-1)}} (1+T^{1-\frac{\gamma}{2}})^{-\frac{1}{p-1}} 
	&&\mbox{ if }p<p_F, \\
	&c_* (1+T^\frac{1}{2}+T^{1-\frac{\gamma}{2}})^{-\frac{1}{p-1}} 
	&&\mbox{ if }p=p_F, \\
	&c_* 
	&&\mbox{ if }p>p_F, 
	\end{aligned} 
	\right.
	\qquad 
	C_0 \leq C_* T^{-\frac{2-\gamma}{2(p-1)}}, 
\]
respectively, then \eqref{eq:main} possesses a solution on $\R^N\times[0,T)$. 
In addition, if $p>p_F$, $c\leq c_*$ and $C_0=0$, 
then \eqref{eq:main} possesses a solution on $\R^N\times[0,\infty)$. 
\end{theorem}

\begin{remark}
We can improve $\psi$ to $|z|^{\gamma/(p-1)}$ for each case 
if $T$ and $z\neq0$ satisfy
\begin{equation}\label{eq:Tzas}
	0<T\leq |z|^2. 
\end{equation} 
The improved statement is optimal up to $c_*$ 
compared with \eqref{eq:OSfjt}. 
Moreover, the assumption \eqref{eq:Tzas} is natural and is not so restrictive, 
since we can take $T$ arbitrarily small 
when we only need to consider local-in-time solvability. 
For more details, 
see Remarks {\rm \ref{rem:imprcri}} and {\rm \ref{rem:imprsubcri}} below. 
\end{remark}

\begin{remark}
In {\rm \cite{STW}}, it was shown that $L^{q_c}(\R^N)$ 
is the scaling invariant critical space 
for the well-posedness of \eqref{eq:main}, 
where $q_c:= N(p-1)/(2-\gamma)$ with $p\geq p_\gamma$. 
For $p>p_\gamma$, they also obtained a solution with initial data under 
$c |\cdot|^{-(2-\gamma)/(p-1)}\in L^{q_c,\infty}(\R^N)$, where 
$c>0$ is small and $L^{q_c,\infty}(\R^N)$ is the Lorentz space. 
Theorem {\rm \ref{th:main}} implies that, 
for $p\geq p_\gamma$, there exists $u_0$ such that 
$u_0$ does not belong to $L^{q_c,\infty}(\R^N)$ 
but \eqref{eq:main} possesses a solution. 
\end{remark}

We next consider the case where there is a possibility that 
the singularity of $u_0$ is located at the origin. 
In this case, the optimal singularity of $u_0$ 
is weaker than that of the Fujita equation.

\begin{theorem}\label{th:origin}
Fix $N\geq1$, $p>1$, $0<\gamma<\min\{2,N\}$, $0<T<\infty$ and 
$z\in \R^N$. 
Assume either $u_0$ is a nonnegative Radon measure satisfying 
\[
	u_0=c \delta_z + \phi \qquad \mbox{ if }p<p_\gamma 
\]
for a nonnegative function $\phi\in L^\infty(\R^N)$ 
with $\|\phi\|_{L^\infty(\R^N)}\leq C_0$, 
or $u_0$ is a nonnegative measurable function satisfying 
\[
	u_0(x)\leq \left\{ 
	\begin{aligned}
	& c |x-z|^{-N} 
	\left( \log ( e+|x-z|^{-1} ) \right)^{-\frac{N}{2-\gamma}-1}
	\chi_1(|x-z|) + C_0 
	&&\mbox{ if }p=p_\gamma, \\
	& c |x-z|^{-\frac{2-\gamma}{p-1}} + C_0 
	&&\mbox{ if }p>p_\gamma 
	\end{aligned}
	\right.
\]
for any $x\in \R^N\setminus\{z\}$. 
Here $c>0$ and $C_0\geq0$ are constants. 
Then there exist positive constants $c_*$ and $C_*$ 
depending on $N$, $p$ and $\gamma$ but not on $T$ and $z$ 
such that the following statements hold. 
If the constants $c$ and $C_0$ satisfy 
\[
	c\leq 
	\left\{
	\begin{aligned} 
	&c_* T^{-\frac{N(p_F-p)}{2(p-1)}} &&\mbox{ if }p<p_\gamma, \\
	&c_* (1+T^\frac{1}{2})^{-\frac{1}{p-1}} &&\mbox{ if }p=p_\gamma, \\
	&c_* 
	&&\mbox{ if }p>p_\gamma, 
	\end{aligned} 
	\right.
	\qquad 
	C_0 \leq C_* T^{-\frac{2-\gamma}{2(p-1)}}, 
\]
respectively, then \eqref{eq:main} possesses a solution on $\R^N\times[0,T)$. 
In addition, if $p>p_\gamma$, $c\leq c_*$ and $C_0=0$, 
then \eqref{eq:main} possesses a solution on $\R^N\times[0,\infty)$. 
\end{theorem}

\begin{remark}
Theorem {\rm \ref{th:origin}} with $z=0$ is optimal up to $c_*$ 
compared with \eqref{eq:OShd}. 
\end{remark}

\begin{remark}
In the case of $u_0=\delta_z$ with $z\in \R^N$. 
Theorems {\rm \ref{th:main}} and {\rm \ref{th:origin}} together with 
{\rm \cite{HS} }imply that \eqref{eq:main} possesses 
a local-in-time solution if and only if either 
$1<p< p_\gamma$, or $p_\gamma\leq p<p_F$ and $z\neq 0$.
\end{remark}

\begin{remark}
The main novelty of Theorem {\rm \ref{th:origin}} 
is the critical case of  $p= p_\gamma$, since 
the case of $p<p_\gamma$ is not difficult and 
the case of  $p>p_\gamma$ follows from {\rm \cite{STW}}. 
However, we give a short proof for $p< p_\gamma$
and a unified proof for $p\geq p_\gamma$. 
\end{remark}

The proofs of our main results are based on showing the existence of 
supersolutions for the corresponding integral equation to \eqref{eq:main}.
Each of the supersolutions is constructed in the same spirit of 
Robinson and Sier\.{z}\polhk ega \cite{RS13}, and has the form
\[
	u^+(x,t):= H^{-1}\left(\int_{\RN} G(x-y, mt) H(u_0(y)) dy\right), 
\]
where $m\in \N$ and 
$H:[0,\infty)\rightarrow [0,\infty)$ is a strictly 
increasing and convex function. 
However, the way to verify that $u^+$ is indeed a supersolution 
is totally different from \cite{RS13}. 
They used uniform estimates with respect to $x$, but 
we estimate $u^+$ uniformly with respect to $t$ 
and analyze its singularity in detail.

The rest of this paper is organized as follows.
In Section \ref{sec:preliminaries}, we give the definition of solutions 
and collect estimates on the heat kernel. 
In Sections \ref{sec:rmain} and \ref{sec:rrrmain}, 
we prove Theorem \ref{th:main} for $p\geq p_F$ and $p<p_F$, respectively.
In Sections \ref{sec:1.2sup} and \ref{sec:1.2sub}, 
we prove Theorem \ref{th:origin} for $p\geq p_\gamma$ and $p<p_\gamma$, respectively.
In Section \ref{sec:frac}, 
we state results for the fractional Hardy parabolic equation.

\section{Preliminaries}\label{sec:preliminaries}
In Subsection \ref{subsec:integ}, we introduce the definition of solutions 
in this paper and quote a lemma on the existence of solutions. 
In Subsection \ref{subsec:heat}, we give estimates concerning the heat kernel.

\subsection{Integral equation}\label{subsec:integ}
Let $u_0$ be a nonnegative measurable function on $\RN$ or 
a nonnegative Radon measure on $\R^N$. 
In what follows, we always consider nonnegative solutions 
of the following integral equation corresponding to \eqref{eq:main}. 
\[
\begin{aligned}
	&u(x,t) = \Phi[u](x,t), \\
	&\Phi[u](x,t) := \int_{\RN} G(x-y,t) du_0(y) 
	+ \int_0^t \int_{\RN} G(x-y, t-s) |y|^{-\gamma} u(y,s)^p dyds, 
\end{aligned}
\]
where $G(x,t):=(4\pi t)^{-N/2} e^{-|x|^2/(4t)}$ 
is the Gaussian heat kernel in $\RN$. 
Remark that, in the case where $u_0$ is a measurable function, 
we regard $du_0(y)$ as $u_0 dy$. 
Solutions and supersolutions are defined as follows. 

\begin{definition}\label{def:sol}
Let $0<T\leq \infty$. 
A nonnegative measurable function $u$ on $\RN\times(0,T)$ 
is called a solution of \eqref{eq:main} on $\RN\times[0,T)$ if 
$u$ satisfies $0\leq u<\infty$ and $u=\Phi[u]$ a.e. in $\RN\times(0,T)$. 
In addition, a nonnegative measurable function $\overline{u}$ on $\RN\times(0,T)$ 
is called a supersolution of \eqref{eq:main} on $\RN\times[0,T)$ if 
$\overline{u}$ satisfies $0\leq \overline{u}<\infty$ 
and $\overline{u}\geq \Phi[\overline{u}]$ a.e. in $\RN\times(0,T)$. 
\end{definition}

The method to construct solutions of \eqref{eq:main} 
is based on the following lemma. 

\begin{lemma}\label{lem:super}
Let $0<T\leq\infty$. 
Assume that there exists a supersolution $\overline{u}$ of 
\eqref{eq:main} on $\RN\times[0,T)$. 
Then there exists a solution on \eqref{eq:main} on $\RN\times[0,T)$. 
\end{lemma}

\begin{proof}
This lemma was proved by \cite[Lemma 2.2]{HS},  see also \cite{HI18,RS13}. 
\end{proof}

\subsection{Estimates on the heat kernel}\label{subsec:heat}

We first give an equality and an estimate for the heat kernel, 
and then we list estimates on some integrals.

\begin{lemma}\label{lem:Gtra}
For any $x, y, \eta\in \R^N$ and $0<s<t$, 
\[
	G(x-y,t-s) G(y-\eta,s)
	=
	G(x-\eta,t) 
	G\left( y-\frac{s}{t}x - \frac{t-s}{t}\eta, 
	\frac{s(t-s)}{t}\right). 
\]
\end{lemma}

\begin{proof}
This can be checked by the straightforward computations. 
\end{proof}

\begin{lemma}\label{lem:Gmid}
There exists a constant $C>0$ depending only on $N$ 
such that, for any $x,y,\eta\in \R^N$ and $0<s<t$, 
\[
\begin{aligned}
	G\left( y-\frac{s}{t}x-\frac{t-s}{t}\eta, 2\frac{s(t-s)}{t} \right) 
	e^{-\frac{|x-y|^2}{8(t-s)}} 
	\leq C  ( G( y-\eta, 35 s) + G( y-x, 35 (t-s) ) ). 
\end{aligned}
\]
\end{lemma}

\begin{proof}
In this proof, we denote by $\tilde G$ the left hand side of the desired inequality. 
Let us first consider the case of $0<s<t/2$. 
By using $|a+b|^2 \geq \delta |a|^2 -\delta(1-\delta)^{-1}|b|^2$ 
for $a,b\in\R^N$ and $0<\delta<1$, we have 
\[
\begin{aligned}
	&\left| y-\frac{s}{t}x-\frac{t-s}{t}\eta \right|^2 
	= 
	\left| y-\eta-\frac{s}{t}(x-\eta) \right|^2 
	\geq 
	\frac{1}{5} |y-\eta|^2 - \frac{1}{4} \left( \frac{s}{t} \right)^2 |x-\eta|^2, \\
	& |x-y|^2 \geq \frac{1}{8}|x-\eta|^2 - \frac{1}{7} |y-\eta|^2, 
\end{aligned}
\]
and so
\[
\begin{aligned}
	\tilde G &\leq 
	C \left( 2 \frac{s(t-s)}{t} \right)^{-\frac{N}{2}} 
	\exp\left( -\frac{5^{-1} |y-\eta|^2}{8 s(t-s) t^{-1} } \right) 
	\exp\left( \frac{ 4^{-1} s^2 t^{-2} |x-\eta|^2}{8s(t-s)t^{-1}} \right) \\
	&\quad 
	\times 
	\exp\left( -\frac{|x-\eta|^2}{64(t-s)} \right)
	\exp\left( \frac{|y-\eta|^2}{56(t-s)} \right)  \\
	&\leq 
	C  s^{-\frac{N}{2}} 
	\exp\left( -\frac{|y-\eta|^2}{40 s} \right) 
	\exp\left( \frac{|x-\eta|^2}{64(t-s)} \right) 
	\exp\left( -\frac{|x-\eta|^2}{64(t-s)}  \right)
	\exp\left( \frac{|y-\eta|^2}{56 s }  \right)   \\
	&= 
	C s^{-\frac{N}{2}} 
	\exp\left( -\frac{ |y-\eta|^2}{140 s } \right)
	= CG(y-\eta,35s)
\end{aligned} 
\]
for any $0<s<t/2$, where $C$ is a constant depending only on $N$. 

We next consider the case of $t/2<s<t$. By 
\[
\begin{aligned}
	\left| y-\frac{s}{t}x-\frac{t-s}{t}\eta \right|^2 
	= 
	\left| y-x+ \frac{t-s}{t} (x-\eta) \right|^2 
	\geq 
	\frac{1}{5} \left( \frac{t-s}{t} \right)^2 |x-\eta|^2 
	- \frac{1}{4} |x-y|^2, 
\end{aligned}
\]
we have 
\[
\begin{aligned}
	\tilde G &\leq 
	C \left( 2 \frac{s(t-s)}{t} \right)^{-\frac{N}{2}} 
	\exp\left( -\frac{5^{-1} (t-s)^2 t^{-2} |x-\eta|^2}
	{8 s(t-s) t^{-1} } \right) 
	\exp\left( \frac{|x-y|^2}{32s(t-s) t^{-1} } \right)
	\exp\left( - \frac{|x-y|^2}{8(t-s)} \right) \\
	&\leq 
	C   (t-s)^{-\frac{N}{2}} 
	\exp\left( \frac{|x-y|^2}{16(t-s)} \right)
	\exp\left( - \frac{|x-y|^2}{8(t-s)} \right) \\
	&= 
	C (t-s)^{-\frac{N}{2}} 
	\exp\left( - \frac{|y-x|^2}{16 (t-s)} \right) 
	\leq  CG(y-x,35(t-s))
\end{aligned} 
\]
for any $t/2<s<t$ with a constant $C$ depending only on $N$. 
Hence the lemma follows. 
\end{proof}

\begin{lemma}\label{lem:linxpoly}
Let $0<k<N$. 
Then there exists $C>0$ depending only on $N$ and $k$ such that, 
for any $x\in\R^N$ and $t>0$, 
\[
	\int_{\R^N} G(x-y,t) |y|^{-k} dy 
	\leq C |x|^{-k}. 
\]
\end{lemma}

\begin{proof}
Set $\Omega_1:=\{y\in\R^N; |y|\leq |x|/2\}$ and 
$\Omega_2:=\{y\in\R^N; |y|\geq |x|/2\}$. 
For $y\in \Omega_1$, we have $|x-y|\geq |x|-|y|\geq |x|/2$. Thus, 
\[
	\int_{\Omega_1} G(x-y,t) |y|^{-k} dy 
	\leq C t^{-\frac{N}{2}} e^{-\frac{|x|^2}{16t}} \int_{\Omega_1} |y|^{-k} dy 
	\leq C t^{-\frac{N}{2}} e^{-\frac{|x|^2}{16t}} |x|^{N-k} 
	\leq C|x|^{-k}, 
\]
since $\sup_{t>0} t^{-N/2} e^{-|x|^2/(16t)}\leq C |x|^{-N}$. 
On the other hand, from $\int_{\R^N} e^{-|x-y|^2/(4t)} dy = C t^{N/2}$, 
it follows that 
\[
	\int_{\Omega_2} G(x-y,t) |y|^{-k} dy 
	\leq C t^{-\frac{N}{2}} |x|^{-k} \int_{\Omega_2} e^{-\frac{|x-y|^2}{4t}} dy 
	\leq C t^{-\frac{N}{2}} |x|^{-k} \int_{\R^N} e^{-\frac{|x-y|^2}{4t}} dy
	= C|x|^{-k}. 
\]
Hence the lemma follows. 
\end{proof}

\begin{lemma}\label{lem:linxlog}
Let $k>0$. 
Then there exists $C>0$ depending only on $N$ and $k$ such that 
\[
	\int_{\R^N} G(x-y,t) |y|^{-N} 
	\left( \log(e+|y|^{-1}) \right)^{-k-1} \chi_1(|y|) dy 
	\leq C |x|^{-N} \left( \log(e+|x|^{-1}) \right)^{-k}
\]
for any $x\in\R^N$ and $t>0$. 
\end{lemma}

\begin{proof}
Set $\Omega_1$ and $\Omega_2$ as in Lemma \ref{lem:linxpoly}. 
Using $|x-y|\geq |x|/2$ for $y\in \Omega_1$ gives 
\[
\begin{aligned}
	&\int_{\Omega_1} G(x-y,t) |y|^{-N} 
	\left( \log(e+|y|^{-1}) \right)^{-k-1} \chi_1(|y|) dy \\
	&\leq C t^{-\frac{N}{2}} e^{-\frac{|x|^2}{16t}} 
	\int_0^{\min\{ |x|/2, 1\}}
	(er+1)\frac{r^{-1}}{er+1} \left( \log(e+r^{-1}) \right)^{-k-1} dy \\
	&\leq C |x|^{-N} 
	\int_0^{\min\{ |x|/2, 1\}}
	\frac{r^{-1}}{er+1} \left( \log(e+r^{-1}) \right)^{-k-1} dy \\
	&\leq C |x|^{-N} \left( \log(e+2|x|^{-1}) \right)^{-k} 
	\leq C |x|^{-N} \left( \log(e+|x|^{-1}) \right)^{-k}. 
\end{aligned}
\]
We fix a constant $K>0$ so large that 
$X\mapsto X^{-N} (\log(K+X^{-1}))^{-k-1}$ is decreasing. 
Then, 
\[
\begin{aligned}
	&\int_{\Omega_2} G(x-y,t) |y|^{-N} 
	\left( \log(e+|y|^{-1}) \right)^{-k-1} dy \\
	&\leq 
	C \int_{\Omega_2} G(x-y,t) |y|^{-N} 
	\left( \log(K+|y|^{-1}) \right)^{-k-1} dy  \\
	&\leq 
	C t^{-\frac{N}{2}}  |x|^{-N} \left( \log(K+2|x|^{-1}) \right)^{-k-1}
	\int_{\Omega_2} e^{-\frac{|x-y|^2}{4t}}  dy \\
	&\leq 
	C  |x|^{-N} \left( \log(K+2|x|^{-1}) \right)^{-k-1} 
	\leq 
	C  |x|^{-N} \left( \log(e+|x|^{-1}) \right)^{-k}, 
\end{aligned}
\]
and the lemma is proved. 
\end{proof}

\begin{lemma}\label{lem:lint}
Let $\varphi\in L^1_\loc(\R^N)$ be a nonnegative function. 
Then there exists $C>0$ depending only on $N$ such that, 
for any $x\in\R^N$ and $t>0$, 
\[
	\int_{\R^N} G(x-y,t) \varphi(y) dy \leq 
	C t^{-\frac{N}{2}} 
	\sup_{\zeta\in\R^N} \int_{B(\zeta,t^\frac{1}{2})} \varphi(y) dy. 
\]
\end{lemma}

\begin{proof}
This follows from \cite[Lemma 2.1]{HI18} with $\theta=2$. 
\end{proof}

\section{Proof of Theorem \ref{th:main} for $p\geq p_F$}\label{sec:rmain}
In this section, we prove Theorem \ref{th:main} for $p\geq p_F$. 
We prepare an auxiliary function $H$, 
and then we define a candidate $\overline{u}$ of a supersolution. 
For $1<\alpha<\min\{p,N(p-1)/2\}$, $0<\beta<N/2$ and $A\geq e$, set
\begin{equation}\label{eq:HX}
	H(X):=
	\left\{ 
	\begin{aligned}
	&X^\alpha &&\mbox{ if } p>p_F, \\
	&X\left( \log(A+X) \right)^\beta &&\mbox{ if } p=p_F. 
	\end{aligned}
	\right.
\end{equation}
We fix $A$ so large that 
\begin{equation}\label{eq:monoX}
\begin{aligned}
	&\mbox{$X\mapsto H(X)$, 
	$X\mapsto X^p/H(X)$ and $X\mapsto H(X)/X$ are strictly increasing and }  \\
	&X\mapsto X^{-(2-\gamma)}\left( \log(A+X^{-1}) \right)^{-1-\beta}
	\mbox{ is strictly decreasing.} 
\end{aligned}
\end{equation}
Note that $H$ is convex and strictly increasing. 
In particular, the inverse function $H^{-1}$ exists and is strictly increasing. 
We can check that $H^{-1}$ satisfies 
\begin{equation}\label{eq:Hinv}
	H^{-1}(X)
	\left\{ 
	\begin{aligned}
	&= X^{1/\alpha} &&\mbox{ if } p>p_F, \\
	&\leq C X\left(\log(A+X)\right)^{-\beta} &&\mbox{ if } p=p_F, 
	\end{aligned}
	\right.
\end{equation}
where $C>0$ depends only on $A$ and $\beta$. 
Let $z\in \R^N\setminus\{0\}$. 
Then, for $c>0$ and $C_0\geq0$, we define 
\[
	\overline{u}(x,t) := 
	2^{\frac{N}{2}+1}c \psi(z) U(x,t) + 2C_0. 
\]
Here $\psi$ is defined by \eqref{eq:psizdef} and 
\[
\begin{aligned}
	&U(x,t):=H^{-1} \left( \int_{\R^N} G(x-y,2t) H(f(y)) dy \right), \\
	& f(x):= 
	\left\{ 
	\begin{aligned}
	& |x-z|^{-\frac{2}{p-1}} 
	&&\mbox{ if }p>p_F, \\
	& |x-z|^{-N} 
	\left( \log ( e+|x-z|^{-1} ) \right)^{-\frac{N}{2}-1} 
	\chi_1(|x-z|)
	&&\mbox{ if }p=p_F.  \\ 
	\end{aligned}
	\right. 
\end{aligned}
\]

We also prepare some estimates of $U$. 
Note that, for $p=p_F$, 
we can check that 
$H(f(y+z)) \leq C|y|^{-N} (\log(e+|y|^{-1}))^{-(N/2)-1+\beta} \chi_1(|y|)$
with a constant $C$ depending on $A$. Thus, 
\[
\begin{aligned}
	H(U(x,t)) &= 
	\int_{\R^N} G(x-z-y,2t) H(f(y+z)) dy \\
	&\leq 
	\left\{ 
	\begin{aligned}
	&\int_{\R^N} G(x-z-y,2t) |y|^{-\frac{2\alpha}{p-1}} dy  &&(p>p_F) \\
	&C \int_{\R^N} G(x-z-y,2t) |y|^{-N} 
	\left( \log(e+|y|^{-1}) \right)^{-\frac{N}{2}-1+\beta}\chi_1(|y|) dy 
	&&(p=p_F). 
	\end{aligned}
	\right. 
\end{aligned}
\]
By Lemmas \ref{lem:linxpoly} and \ref{lem:linxlog}, we have 
\[
\begin{aligned}
	H(U(x,t))&\leq 
	\left\{ 
	\begin{aligned}
	&C |x-z|^{-\frac{2\alpha}{p-1}}  &&\mbox{ if }p>p_F, \\
	&C |x-z|^{-N} \left( \log(e+|x-z|^{-1}) \right)^{-\frac{N}{2}+\beta} 
	&&\mbox{ if }p=p_F. 
	\end{aligned}
	\right. 
\end{aligned}
\]
Then the monotonicity of $H^{-1}$ together with \eqref{eq:Hinv} implies that 
\begin{equation}\label{eq:Uesti}
	U(x,t) \leq 
	\left\{ 
	\begin{aligned}
	&C |x-z|^{-\frac{2}{p-1}}  &&\mbox{ if }p>p_F, \\
	&C |x-z|^{-N} \left( \log(e+|x-z|^{-1}) \right)^{-\frac{N}{2}} 
	&&\mbox{ if }p=p_F. 
	\end{aligned}
	\right. 
\end{equation}

On the other hand, 
Lemma \ref{lem:lint}  yields
\[
\begin{aligned}
	H(U(x,t))
	&\leq 
	\left\{ 
	\begin{aligned}
	&Ct^{-\frac{N}{2}} 
	\sup_{\zeta\in\R^N} \int_{B(\zeta,(2t)^\frac{1}{2})} 
	|y|^{-\frac{2\alpha}{p-1}} dy  &&(p>p_F) \\
	&Ct^{-\frac{N}{2}} 
	\sup_{\zeta\in\R^N} \int_{B(\zeta,(2t)^\frac{1}{2})} 
	|y|^{-N} \left( \log(e+|y|^{-1}) \right)^{-\frac{N}{2}-1+\beta} \chi_1(|y|) dy 
	&&(p=p_F) 
	\end{aligned}
	\right. \\
	&\leq 
	\left\{ 
	\begin{aligned}
	&Ct^{-\frac{N}{2}} 
	\int_{B(0,(2t)^\frac{1}{2})} 
	|y|^{-\frac{2\alpha}{p-1}} dy  &&(p>p_F) \\
	&Ct^{-\frac{N}{2}} 
	\int_{B(0,(2t)^\frac{1}{2})} 
	(e|y|+1)\frac{|y|^{-N}}{e|y|+1} 
	\left( \log(e+|y|^{-1}) \right)^{-\frac{N}{2}-1+\beta} \chi_1(|y|)dy 
	&&(p=p_F)
	\end{aligned}
	\right. \\
	&\leq 
	\left\{ 
	\begin{aligned}
	&C t^{-\frac{\alpha}{p-1}}  &&(p>p_F) \\
	&C t^{-\frac{N}{2}} 
	\left( \log(e+t^{-\frac{1}{2}}) \right)^{-\frac{N}{2}+\beta} 
	&&(p=p_F). 
	\end{aligned}
	\right. 
\end{aligned}
\]
Then the monotonicity of $H^{-1}$ together with \eqref{eq:Hinv} implies that 
\begin{equation}\label{eq:Uestiuni}
	U(x,t) \leq 
	\left\{ 
	\begin{aligned}
	&C t^{-\frac{1}{p-1}}  &&\mbox{ if }p>p_F, \\
	&C t^{-\frac{N}{2}} \left( \log(e+t^{-\frac{1}{2}}) \right)^{-\frac{N}{2}} 
	&&\mbox{ if }p=p_F. 
	\end{aligned}
	\right. 
\end{equation}

We are now in a position to prove Theorem \ref{th:main}.

\begin{proof}[Proof of Theorem \ref{th:main} for $p\geq p_F$]
By Lemma \ref{lem:super}, 
it suffices to prove that $\overline{u}$ is a supersolution 
if $c$ is sufficiently small. 
First, we consider the case of $C_0=0$ and 
estimate $\Phi[\overline{u}]$. 
Remark that $u_0\leq c \psi(z) f$ with $c>0$. 
This together with Jensen's inequality gives 
\begin{equation}\label{eq:Jens}
\begin{aligned}
	\int_{\R^N} G(x-y,t) u_0(y) dy &\leq 
	c 2^\frac{N}{2} \psi(z) \int_{\R^N} G(x-y,2t) f(y) dy \\
	&= c2^\frac{N}{2} \psi(z) H^{-1} \circ H\left( 
	\frac{\int_{\R^N} G(x-y,2t) f(y) dy }{\int_{\R^N} G(x-y,2t) dy }  \right) \\
	&\leq c2^\frac{N}{2} \psi(z) U(x,t) 
	= \frac{1}{2} \overline{u}(x,t). 
\end{aligned}
\end{equation}
By Fubini's theorem and Lemma \ref{lem:Gtra}, we have 
\begin{equation}\label{eq:Jeq}
\begin{aligned}
	&\int_0^t \int_{\R^N} G(x-y,t-s) |y|^{-\gamma} \overline{u}(y,s)^p  dyds \\
	&= 
	c^p 2^{\frac{N}{2}(p+1)+p} \psi(z)^p
	\int_0^t \int_{\R^N} 
	G(x-y,2(t-s)) e^{-\frac{|x-y|^2}{8(t-s)}} 
	 |y|^{-\gamma} \frac{U(y,2s)^p}{H(U(y,2s))} \\
	&\quad \times \int_{\R^N}  G(y-\eta,2s)  H(f(\eta)) d\eta dyds  \\
	&= 
	c^p 2^{\frac{N}{2}(p+1)+p} \psi(z)^p 
	\int_{\R^N}  G(x-\eta,2t) H(f(\eta)) 
	J(x,\eta,t) d\eta, 
\end{aligned}
\end{equation}
where 
\[
	J(x,\eta,t):= 
	\int_0^t \int_{\R^N}  
	G\left(y-\frac{s}{t}x-\frac{t-s}{t}\eta, 2 \frac{s(t-s)}{t} \right) 
	e^{-\frac{|x-y|^2}{8(t-s)} }
	|y|^{-\gamma} \frac{U(y,2s)^p}{H(U(y,2s))} dyds. 
\]

We claim that 
\begin{equation}\label{eq:Jcla}
	J(x,\eta,t) \leq 
	\left\{ 
	\begin{aligned}
	&C \psi(z)^{-(p-1)} t^\frac{\alpha-1}{p-1} 
	&&\mbox{ if }p>p_F, \\
	& C \psi(z)^{-(p-1)} (1+t^\frac{1}{2}+t^{1-\frac{\gamma}{2}}) 
	\left( \log(e+t^{-\frac{1}{2}}) \right)^{-\beta}
	&&\mbox{ if }p=p_F
	\end{aligned}
	\right.  
\end{equation}
for any $t>0$ with a constant $C>0$ independent of $z$. 
Set $\Omega_1:=\{y\in\R^N; |y|\leq |z|/2\}$ and 
$\Omega_2:=\{y\in\R^N; |y|\geq |z|/2\}$. 
Lemma \ref{lem:Gmid} shows that 
\[
\begin{aligned}
	J &\leq 
	C \int_0^t \int_{\R^N}  
	( G( y-\eta, 35 s) + G( y-x, 35 (t-s) ) )
	|y|^{-\gamma} \frac{U(y,2s)^p}{H(U(y,2s))} dyds \\
	&= C \int_0^t \int_{\Omega_1} + C\int_0^t \int_{\Omega_2} 
	=: C J_1 + C J_2. 
\end{aligned}
\]
In what follows, we write 
$\|U(\cdot,t)\|_\infty:= \|U(\cdot,t)\|_{L^\infty(\R^N)}$ and 
$|U(x,\cdot)|_\infty:=|U(x,\cdot)|_{L^\infty((0,\infty))}$.

Let us first estimate $J_1$. 
The monotonicity of $X\mapsto X^p/H(X)$ in \eqref{eq:monoX} yields 
\[
\begin{aligned}
	J_1 &\leq 
	C \int_0^t \int_{\Omega_1}  
	( G( y-\eta, 35 s) + G( y-x, 35 (t-s) )
	|y|^{-\gamma} \frac{|U(y,\cdot)|_\infty^p}{H(|U(y,\cdot)|_\infty)} dyds. 
\end{aligned}
\]
From \eqref{eq:Uesti}, $|y-z|\geq |y|$ for $y\in \Omega_1$, 
$N(p_F-1)=2$ and the monotonicity of 
$X\mapsto X^{-(2-\gamma)}(\log(A+X^{-1}))^{-1-\beta}$ in \eqref{eq:monoX}, 
it follows that 
\[
\begin{aligned}
	\frac{| U(y,\cdot)|_\infty^p}{H(|U(y,\cdot)|_\infty)}  
	&\leq 
	\left\{ 
	\begin{aligned}
	&C |y-z|^{-\frac{2(p-\alpha)}{p-1}} &&(p>p_F) \\
	&C |y-z|^{-2} \left( \log(A+|y-z|^{-1}) \right)^{-1-\beta}
	&&(p=p_F) 
	\end{aligned}
	\right.  \\
	&\leq 
	\left\{ 
	\begin{aligned}
	&C |z|^{-\gamma} |y|^{-(\frac{2(p-\alpha)}{p-1}-\gamma)} &&(p>p_F) \\
	&C |z|^{-\gamma} |y|^{-(2-\gamma)} \left( \log(e+|y|^{-1}) \right)^{-1-\beta}
	&&(p=p_F, N>2) \\
	&C |z|^{-2}
	&&(p=p_F, N\leq2) 
	\end{aligned}
	\right.   
\end{aligned}
\]
for $y\in \Omega_1$, and so 
\[
\begin{aligned}
	J_1	&\leq 
	\left\{ 
	\begin{aligned}
	&C |z|^{-\gamma}  \sup_{\xi\in\R^N}
	\int_0^t \int_{\Omega_1}  
	G( y-\xi, 35 s) 
	|y|^{-\frac{2(p-\alpha)}{p-1}} dyds
	&&\mbox{ if }p>p_F,  \\
	&C |z|^{-\gamma}  \sup_{\xi\in\R^N} \int_0^t \int_{\Omega_1}  
	G( y-\xi, 35 s) 
	\frac{|y|^{-2}}{ \left( \log(e+|y|^{-1}) \right)^{1+\beta} }
	dyds
	&&\mbox{ if }p=p_F, N>2, \\
	&C |z|^{-2}  \sup_{\xi\in\R^N} \int_0^t \int_{\Omega_1}  
	G( y-\xi, 35 s) |y|^{-\gamma} dyds
	&&\mbox{ if }p=p_F, N\leq2.  
	\end{aligned}
	\right.   
\end{aligned}
\]
Then, Lemma \ref{lem:lint} and straightforward computations show that, 
for any $t>0$, 
\begin{equation}\label{eq:J1escal}
\begin{aligned}
	J_1 &\leq
	\left\{ 
	\begin{aligned}
	&C |z|^{- \gamma} \int_0^t s^{-\frac{p-\alpha}{p-1}}  ds &&(p>p_F) \\
	&C |z|^{- \gamma} \int_0^t (es^\frac{1}{2}+1) 
	\frac{s^{-1}}{es^{1/2}+1}\left( \log(e+s^{-\frac{1}{2}}) \right)^{-1-\beta}  ds
	&&(p=p_F, N>2) \\
	&C |z|^{-2} \int_0^t s^{-\frac{\gamma}{2}} ds
	&&(p=p_F, N\leq2)
	\end{aligned}
	\right.   \\
	&\leq
	\left\{ 
	\begin{aligned}
	&C |z|^{-\gamma} t^\frac{\alpha-1}{p-1}  ds &&(p>p_F) \\
	&C |z|^{-\gamma} (1+t^\frac{1}{2})  \left( \log(e+t^{-\frac{1}{2}}) \right)^{-\beta} 
	&&(p=p_F, N\geq3) \\
	&C |z|^{-2} t^{1-\frac{\gamma}{2}} &&(p=p_F, N=1,2). 
	\end{aligned}
	\right.   
\end{aligned}
\end{equation}
Note that 
\[
	|z|^{-2} t^{1-\frac{\gamma}{2}}
	\leq C|z|^{-2}(1+|z|)^{2-\gamma} (1+t^\frac{1}{2}+t^{1-\frac{\gamma}{2}}) 
	\left( \log(e+t^{-\frac{1}{2}}) \right)^{-\beta}
\]
for any $t>0$, where this constant $C>0$ is independent of $z$.

We next examine $J_2$. 
By using \eqref{eq:Uestiuni}, we can check that 
\[
\begin{aligned}
	\frac{\| U(\cdot,2s)\|_\infty^p}{H(\|U(\cdot,2s)\|_\infty)}  
	&\leq 
	\left\{ 
	\begin{aligned}
	&C s^{-\frac{p-\alpha}{p-1}} &&\mbox{ if }p>p_F, \\
	&C s^{-1} \left( \log(e+s^{-\frac{1}{2}}) \right)^{-1-\beta}
	&&\mbox{ if }p=p_F. 
	\end{aligned}
	\right.  
\end{aligned}
\]
This together with $\int_{\R^N} G dy =1$ yields 
\[
\begin{aligned}
	J_2 &\leq 
	\left\{ 
	\begin{aligned}
	&C |z|^{-\gamma}\int_0^t s^{-\frac{p-\alpha}{p-1}} ds
	&&(p>p_F) \\
	& C |z|^{-\gamma} \int_0^t 
	(es^{1/2}+1)
	\frac{s^{-1}}{es^{1/2}+1} \left( \log(e+s^{-\frac{1}{2}}) \right)^{-1-\beta} ds
	&&(p=p_F)
	\end{aligned}
	\right.   \\
	&\leq 
	\left\{ 
	\begin{aligned}
	&C |z|^{-\gamma}t^\frac{\alpha-1}{p-1} 
	&&(p>p_F) \\
	& C |z|^{-\gamma} (1+t^\frac{1}{2}) 
	\left( \log(e+t^{-\frac{1}{2}}) \right)^{-\beta} 
	&&(p=p_F). 
	\end{aligned}
	\right.
\end{aligned}
\]
Note that 
\[
	|z|^{-\gamma} (1+t^\frac{1}{2}) \left( \log(e+t^{-\frac{1}{2}}) \right)^{-\beta} 
	\leq C|z|^{-2}(1+|z|)^{2-\gamma} (1+t^\frac{1}{2}+t^{1-\frac{\gamma}{2}}) 
	\left( \log(e+t^{-\frac{1}{2}}) \right)^{-\beta}
\]
for any $t>0$ with a constant $C>0$ independent of $z$. 
Therefore, the above estimates 
show the desired inequality \eqref{eq:Jcla}.

The equality \eqref{eq:Jeq} and the estimate \eqref{eq:Jcla} imply that 
\[
\begin{aligned}
	&\int_0^t \int_{\R^N} G(x-y,t-s) |y|^{-\gamma} \overline{u}(y,s)^p  dyds \\
	&\leq  
	\left\{ 
	\begin{aligned}
	&C c^p \psi(z) t^\frac{\alpha-1}{p-1}
	\int_{\R^N}  G(x-\eta,2t) H(f(\eta))  d\eta 
	&&\mbox{ if }p>p_F,  \\
	&C c^p \psi(z) 
	(1+t^\frac{1}{2}+t^{1-\frac{\gamma}{2}}) 
	\left( \log(e+ t^{-\frac{1}{2}}) \right)^{-\beta}
	\int_{\R^N}  G(x-\eta,2t) H(f(\eta))  d\eta 
	&&\mbox{ if }p=p_F. 
	\end{aligned}
	\right. 
\end{aligned}
\]
By the monotonicity of $X\mapsto H(X)/X$ in \eqref{eq:monoX}, we have 
\begin{equation}\label{eq:GHflin}
	\int_{\R^N}  G(x-\eta,2t) H(f(\eta)) d\eta
	= H(U(x,t)) 
	\leq 
	\frac{H(\|U(\cdot,t)\|_\infty)}{\|U(\cdot,t)\|_\infty} U(x,t). 
\end{equation}
Then \eqref{eq:Uestiuni} gives 
\[
\begin{aligned}
	\int_{\R^N}  G(x-\eta,2t) H(f(\eta)) d\eta
	&\leq
	\left\{ 
	\begin{aligned}
	&C \|U(\cdot,t)\|_\infty^{\alpha-1} U(x,t) &&(p>p_F) \\
	&C \left(  \log( A+ \|U(\cdot,t)\|_\infty ) \right)^\beta 
	U(x,t) &&(p=p_F) \\
	\end{aligned}
	\right. \\
	&\leq
	\left\{ 
	\begin{aligned}
	&C t^{-\frac{\alpha-1}{p-1}} U(x,t) &&(p>p_F) \\
	&C  \left( \log( e+ t^{-\frac{1}{2}} )\right)^\beta 
	U(x,t) &&(p=p_F),  \\
	\end{aligned}
	\right. \\
\end{aligned}
\]
and so 
\begin{equation}\label{eq:uUcal}
\begin{aligned}
	\int_0^t \int_{\R^N} G(x-y,t-s) |y|^{-\gamma} \overline{u}(y,s)^p  dyds 
	&\leq 
	\left\{
	\begin{aligned}
	&C c^p \psi(z) U(x,t)  
	&&(p>p_F) \\
	&C c^p \psi(z) (1+t^\frac{1}{2}+t^{1-\frac{\gamma}{2}})  U(x,t) 
	&&(p=p_F) \\
	\end{aligned}
	\right. \\
	&= 
	\left\{
	\begin{aligned}
	&C c^{p-1} \overline{u}(x,t)  
	&&(p>p_F) \\
	&C c^{p-1} (1+t^\frac{1}{2}+t^{1-\frac{\gamma}{2}})  \overline{u}(x,t) &&(p=p_F). 
	\end{aligned}
	\right. 
\end{aligned}
\end{equation}
Then, in the case of $C_0=0$, the above computations imply that 
\[
\begin{aligned}
	\Phi[\overline{u}]
	&\leq 
	\left\{
	\begin{aligned}
	&(2^{-1} + C c^{p-1}) \overline{u}  
	&&\mbox{ if }p>p_F, \\
	&(2^{-1} + C c^{p-1} (1+t^\frac{1}{2}+t^{1-\frac{\gamma}{2}})  ) \overline{u} 
	&&\mbox{ if }p=p_F 
	\end{aligned}
	\right. 
\end{aligned}
\]
for any $t>0$, where $C>0$ is a constant depending on $N$, $p$ and $\gamma$ 
but not on $t$ and $z$. 
In particular, for $p>p_F$, 
$\overline{u}$ is a supersolution on $\R^N\times [0,\infty)$ 
if $c$ is sufficiently small depending only on $N$, $p$ and $\gamma$. 
This together with Lemma \ref{lem:super} shows that 
\eqref{eq:main} possesses a solution on $\R^N\times[0,\infty)$ 
provided that $p>p_F$, $C_0=0$ and $c$ is small enough.

Let us next consider the case of $C_0\geq0$. 
Fix $0<T<\infty$. 
By Lemma \ref{lem:lint} and the same computations as 
\eqref{eq:Jens} and \eqref{eq:uUcal}, we have 
\[
\begin{aligned}
	\Phi[\overline{u}] 
	&\leq 
	\left\{
	\begin{aligned}
	&2^{\frac{N}{2}+1} c \psi(z) 
	(2^{-1}+\tilde C c^{p-1})U
	+(1+\tilde C C_0^{p-1} T^{1-\frac{\gamma}{2}} )C_0
	&&\mbox{ if }p>p_F, \\
	&2^{\frac{N}{2}+1} c  \psi(z)
	(2^{-1}+\tilde C (1+T^\frac{1}{2}+T^{1-\frac{\gamma}{2}})c^{p-1}) U
	+(1+\tilde C  C_0^{p-1} T^{1-\frac{\gamma}{2}} )C_0
	&&\mbox{ if }p=p_F  \\
	\end{aligned}
	\right.
\end{aligned}
\]
for any $0<t<T$, where $\tilde C >0$ 
is a constant depending on $N$, $p$ and $\gamma$ 
but not on $T$ and $z$. 
If $c$ and $C_0$ satisfy 
\[
	c\leq \left\{ 
	\begin{aligned}
	& (2\tilde C)^{-\frac{1}{p-1}} &&\mbox{ if }p>p_F, \\
	& (2\tilde C)^{-\frac{1}{p-1}}
	(1+T^\frac{1}{2}+T^{1-\frac{\gamma}{2}})^{-\frac{1}{p-1}}&&\mbox{ if }p=p_F, \\
	\end{aligned}
	\right.
	\qquad 
	C_0\leq \tilde C^{-\frac{1}{p-1}} T^{-\frac{2-\gamma}{2(p-1)}}, 
\]
then $\overline{u}$ is a supersolution on $\R^N\times [0,T)$. 
By Lemma \ref{lem:super}, we obtain a solution on $\R^N\times [0,T)$. 
The proof is complete. 
\end{proof}

\begin{remark}\label{rem:imprcri}
We consider the case of  $p=p_F$ under the additional assumption 
that $0<T<\infty$ and $z\neq0$ satisfy \eqref{eq:Tzas}. 
For $0<t<T \, (\leq |z|^2)$, 
we can improve the case of  $N\leq2$ in \eqref{eq:J1escal} 
to $J_1\leq C|z|^{-\gamma}$, and then we have 
\[
	J(x,\eta,t) \leq 
	C |z|^{-\gamma} (1+t^\frac{1}{2}) 
	\left( \log(e+t^{-\frac{1}{2}}) \right)^{-\beta}
\]
for any $0<t<T$ with a constant $C>0$ independent of $T$ and $z$. 
Thus, by replacing $\psi$ with $|z|^{\gamma/(p-1)}$ for each case, 
we see that the function 
$\overline{u}^+(x,t) := 2^{(N/2)+1}c |z|^{\gamma/(p-1)} U(x,t) + 2C_0$
satisfies 
\[
	\Phi[\overline{u}^+]
	\leq 
	(2^{-1} + C c^{p-1} (1+T^\frac{1}{2})  ) \overline{u}^+ 
	\qquad \mbox{ for any }0<t<T. 
\]
This improves the condition on $\psi$ and $c$ 
in the statement of Theorem {\rm \ref{th:main}} for $p=p_F$. 
We note that the condition on $c$ is improved to 
$c\leq c_* (1+T^{1/2})^{-{1/(p-1)}}$ for $p=p_F$. 
\end{remark}

\section{Proof of Theorem \ref{th:main} for $p< p_F$}\label{sec:rrrmain}
Let $z\in \R^N\setminus\{0\}$ and $\phi\in L^\infty(\R^N)$. 
For $c>0$, we set 
\[
	\overline{w}(x,t) := 2^{\frac{N}{2}+1} c \psi(z) G(x-z,2t) 
	+ 2C_0, 
\]
where $C_0 :=\|\phi\|_{L^\infty(\R^N)}$.

\begin{proof}[Proof of Theorem \ref{th:main} for $p<p_F$]
We check that $\overline{w}$ is a supersolution of \eqref{eq:main} 
if $c$ is small. 
The assumption on $u_0$ gives 
\[
	\int_{\RN} G(x-y,t)d u_0 (y) 
	\leq c \psi(z) G(x-z,t) +C_0
	\leq \frac{1}{2} \overline{w}(x,t). 
\]
By similar computations to \eqref{eq:Jeq} and Lemma \ref{lem:lint}, 
we have 
\[
\begin{aligned}
	&\int_0^t \int_{\R^N} G(x-y,t-s) |y|^{-\gamma} \overline{w}(y,s)^p dyds 
	\leq C c^p \psi(z)^p \overline{J} G(x-z,2t) + 
	CC_0^p t^{1-\frac{\gamma}{2}}, \\ 
	&\overline{J}(x,t):= 
	\int_0^t \int_{\R^N} 
	G\left(y-\frac{s}{t}x-\frac{t-s}{t}z, 2 \frac{s(t-s)}{t} \right) 
	e^{-\frac{|x-y|^2}{8(t-s)} }
	|y|^{-\gamma} G(y-z,2s)^{p-1} dyds. 
\end{aligned}
\]

We estimate $\overline{J}$.
Set $\Omega_1:=\{y\in\R^N; |y|\leq |z|/2\}$, 
$\Omega_2:=\{y\in\R^N; |y|\geq |z|/2\}$ and 
$\overline{J}=\int_0^t\int_{\Omega_1}+\int_0^t\int_{\Omega_2}
=:\overline{J}_1+\overline{J}_2$. 
For $y\in \Omega_1$, we see that 
\[
	G(y-z,2s) = (8\pi)^{-\frac{N}{2}} s^{\frac{1}{p-1}-\frac{N}{2}} 
	s^{-\frac{1}{p-1}} e^{-\frac{|y-z|^2}{8s}} 
	\leq Cs^{\frac{1}{p-1}-\frac{N}{2}} |y-z|^{-\frac{2}{p-1}} 
	\leq Cs^{\frac{1}{p-1}-\frac{N}{2}} |z|^{-\frac{2}{p-1}}. 
\]
This together with Lemmas \ref{lem:Gmid} and \ref{lem:lint} yields 
\begin{equation}\label{eq:J1bar}
\begin{aligned}
	\overline{J}_1&\leq 
	C t^{1-\frac{N}{2}(p-1)} |z|^{-2}
	\int_0^t \int_{\Omega_1} ( G( y-z, 35 s) + G( y-x, 35 (t-s) ))
	|y|^{-\gamma} dyds \\
	&\leq 
	C t^{1-\frac{N}{2}(p-1)+1-\frac{\gamma}{2}} |z|^{-2} 
	\leq  
	C t^{\frac{N}{2}(p_F-p)}(1+t^{1-\frac{\gamma}{2}}) \psi(z)^{-(p-1)}
\end{aligned}
\end{equation}
for any $t>0$. 
On the other hand, by $G(y-z,2s)\leq Cs^{-N/2}$ and $\int_{\R^N} G dy=1$, 
we have 
\[
\begin{aligned}
	\overline{J}_2&\le 
	\int_0^t \int_{\Omega_2} 
	G\left(y-\frac{s}{t}x-\frac{t-s}{t}z, 2 \frac{s(t-s)}{t} \right) 
	|y|^{-\gamma} G(y-z,2s)^{p-1} dyds \\
	&\leq 
	C|z|^{-\gamma}\int_0^t s^{-\frac{N}{2}(p-1)}\int_{\Omega_2} 
	G\left(y-\frac{s}{t}x-\frac{t-s}{t}z, 2 \frac{s(t-s)}{t} \right) dyds \\
	&\leq C|z|^{-\gamma} t^{\frac{N}{2}(p_F-p)}
	\leq 
	C t^{\frac{N}{2}(p_F-p)}(1+t^{1-\frac{\gamma}{2}}) \psi(z)^{-(p-1)}
\end{aligned}
\]
for any $t>0$. Hence 
$\overline{J}(x,t)\leq 
C t^{(N/2)(p_F-p)}(1+t^{1-(\gamma/2)}) \psi(z)^{-(p-1)}$ 
for $t>0$. 

Fix $0<T<\infty$. 
The above computations imply that
\[
	\Phi[\overline{w}] 
	\leq 
	\frac{1}{2} \overline{w} + 
	\tilde C c^{p-1} T^{\frac{N}{2}(p_F-p)} (1+T^{1-\frac{\gamma}{2}}) 
	c\psi(z) G(x-z,2t)
	+ \tilde C C_0^p T^{1-\frac{\gamma}{2}}  
\]
for any $0<t<T$, 
where $\tilde C>0$ is a constant independent of $T$ and $z$. 
If $c$ and $C_0$ satisfy 
\[
	c\leq (2\tilde C)^{-\frac{1}{p-1}}
	T^{-\frac{N(p_F-p)}{2(p-1)}} (1+T^{1-\frac{\gamma}{2}})^{-\frac{1}{p-1}}, 
	\qquad 
	C_0\leq (2\tilde C)^{-\frac{1}{p-1}} T^{-\frac{2-\gamma}{2(p-1)}}, 
\]
then $\overline{w}$ is a supersolution on $\R^N\times [0,T)$. 
By Lemma \ref{lem:super}, we obtain a solution on $\R^N\times [0,T)$. 
The proof is complete. 
\end{proof}

\begin{remark}\label{rem:imprsubcri}
We consider the case of  $p<p_F$ under \eqref{eq:Tzas}. 
For $0<t<T \, (\leq |z|^2)$, we can improve \eqref{eq:J1bar} 
to $J_1\leq Ct^{(N/2)(p_F-p)}|z|^{-\gamma}$, and then we have 
$\overline{J}(x,t) \leq Ct^{(N/2)(p_F-p)} |z|^{-\gamma}$ 
for any $0<t<T$ with a constant $C>0$ independent of $T$ and $z$. 
Thus, we see that the function 
$\overline{w}^+(x,t) := 2^{(N/2)+1}c |z|^{\gamma/(p-1)} G(x-z,2t) + 2C_0$
satisfies 
\[
	\Phi[\overline{w}^+]
	\leq 
	\frac{1}{2} \overline{w}^+ + 
	C c^{p-1} T^{\frac{N}{2}(p_F-p)} 
	c|z|^\frac{\gamma}{p-1} G(x-z,2t)
	+ C C_0^p T^{1-\frac{\gamma}{2}} 
	\mbox{ for any }0<t<T. 
\]
This improves the condition on $\psi$ and $c$ 
in the statement of Theorem {\rm \ref{th:main}} for $p<p_F$. 
The improved condition on $c$ is $c\leq c_* T^{-N(p_F-p)/(2(p-1))}$ for $p<p_F$. 
\end{remark}

\section{Proof of Theorem~\ref{th:origin} for $p\geq p_\gamma$} \label{sec:1.2sup}
Theorem \ref{th:origin} with $p>p_\gamma$ was proved in \cite{STW,HS}. 
However, we handle $p>p_\gamma$ and $p=p_\gamma$ in a unified way. 
Let $H$ be as in \eqref{eq:HX}, 
where $A$ is chosen so large that \eqref{eq:monoX} holds. 
Fix $z\in\R^N$. 
For $c>0$, define 
\[
	\overline{v}(x,t) := 2c V(x,t) + 2C_0, 
\]
where
\[
\begin{aligned}
	&V(x,t):=H^{-1} \left( \int_{\R^N} G(x-y,t) H(g(y)) dy \right), \\
	& g(x):= 
	\left\{ 
	\begin{aligned}
	& |x-z|^{-\frac{2-\gamma}{p-1}} 
	&&\mbox{ if }p>p_\gamma, \\
	& |x-z|^{-N} 
	\left( \log \left( e+|x-z|^{-1} \right) \right)^{-\frac{N}{2-\gamma}-1} 
	\chi_1(|x-z|)
	&&\mbox{ if }p=p_\gamma. 
	\end{aligned}
	\right.
\end{aligned}
\]

We give estimates of $V$. 
Lemma \ref{lem:lint}  shows that 
\[
\begin{aligned}
	H(V(x,t))
	&\leq 
	\left\{ 
	\begin{aligned}
	&Ct^{-\frac{N}{2}} 
	\int_{B(0,t^\frac{1}{2})} 
	|y|^{-\frac{(2-\gamma)\alpha}{p-1}} dy  &&(p>p_\gamma) \\
	&Ct^{-\frac{N}{2}} 
	\int_{B(0,t^\frac{1}{2})} 
	(e|y|+1)\frac{|y|^{-N}}{e|y|+1} 
	\left( \log(e+|y|^{-1}) \right)^{-\frac{N}{2-\gamma}-1+\beta} \chi_1(|y|)dy 
	&&(p=p_\gamma)
	\end{aligned}
	\right. \\
	&\leq 
	\left\{ 
	\begin{aligned}
	&C t^{-\frac{(2-\gamma)\alpha}{2(p-1)}}  &&(p>p_\gamma) \\
	&C t^{-\frac{N}{2}} 
	\left( \log(e+t^{-\frac{1}{2}}) \right)^{-\frac{N}{2-\gamma}+\beta} 
	&&(p=p_\gamma). 
	\end{aligned}
	\right. 
\end{aligned}
\]
Then the monotonicity of $H^{-1}$ together with \eqref{eq:Hinv} implies that 
\begin{equation}\label{eq:Vestiuni}
	V(x,t) \leq 
	\left\{ 
	\begin{aligned}
	&C t^{-\frac{2-\gamma}{2(p-1)}}  &&\mbox{ if }p>p_\gamma, \\
	&C t^{-\frac{N}{2}} 
	\left( \log(e+t^{-\frac{1}{2}}) \right)^{-\frac{N}{2-\gamma}} 
	&&\mbox{ if }p=p_\gamma. 
	\end{aligned}
	\right. 
\end{equation}

We are now in a position to prove Theorem \ref{th:origin}.

\begin{proof}[Proof of Theorem \ref{th:origin} for $p\geq p_\gamma$]
We only give a proof in the case of  $C_0=0$, 
since the case of  $C_0>0$ can be handled in the same way 
as in the last part of Section \ref{sec:rmain}. 
We check that $\overline{v}$ is a supersolution of \eqref{eq:main} 
if $c$ is small. 
By the same computations 
as \eqref{eq:Jens} and \eqref{eq:Jeq}, we have 
\begin{equation}\label{eq:Jtileq} 
\begin{aligned} 
	& \int_{\R^N} G(x-y,t) u_0(y) dy \leq 
	c \int_{\R^N} G(x-y,t) g(y) dy \leq c V(x,t) 
	= \frac{1}{2} \overline{v}(x,t), \\
	& \int_0^t \int_{\R^N} G(x-y,t-s) |y|^{-\gamma} \overline{v}(y,s)^p  dyds 
	= 2^p c^p \int_{\R^N}  G(x-\eta,t) H(g(\eta)) 
	\tilde J(x,\eta,t) d\eta, 
\end{aligned}
\end{equation}
where 
\[
	\tilde J(x,\eta,t):= 
	\int_0^t \int_{\R^N}  
	G\left(y-\frac{s}{t}x-\frac{t-s}{t}\eta, \frac{s(t-s)}{t} \right) 
	|y|^{-\gamma} \frac{V(y,s)^p}{H(V(y,s))} dyds. 
\]

In what follows, we write 
$\|V(\cdot,t)\|_\infty:= \|V(\cdot,t)\|_{L^\infty(\R^N)}$. 
Then the monotonicity of $X\mapsto X^p/H(X)$ in \eqref{eq:monoX} implies that 
\[
\begin{aligned}
	\tilde J(x,\eta,t) &\leq 
	\int_0^t \frac{\|V(\cdot,s)\|_\infty^p}{H(\|V(\cdot,s)\|_\infty)}
	\int_{\R^N} G\left(y-\frac{s}{t}x-\frac{t-s}{t}\eta, \frac{s(t-s)}{t} \right) 
	|y|^{-\gamma}  dyds. 
\end{aligned}
\]
Note that 
\begin{equation}\label{eq:Gxi}
\begin{aligned}
	\sup_{\xi\in\R^N} \int_{\R^N} G\left(y-\xi, \frac{s(t-s)}{t} \right) 
	|y|^{-\gamma}  dy 
	&= 
	\int_{\R^N} G\left(y, \frac{s(t-s)}{t} \right) |y|^{-\gamma}  dy \\
	&\leq C\left(\frac{s(t-s)}{t}\right)^{-\frac{\gamma}{2}}. 
\end{aligned}
\end{equation}
Moreover, by \eqref{eq:Vestiuni}, we have 
\[
	\frac{\|V(\cdot,s)\|_\infty^p}{H(\|V(\cdot,s)\|_\infty)}
	\leq 
	\left\{
	\begin{aligned}
	&Cs^{-\frac{(2-\gamma)(p-\alpha)}{2(p-1)}}&&\mbox{ if }p>p_\gamma,  \\
	&Cs^{-\frac{N}{2}(p-1)} \left( \log(e+s^{-\frac{1}{2}} ) \right)^{-1-\beta}
	&&\mbox{ if }p=p_\gamma. 
	\end{aligned}
	\right. 
\]
Recall $p_\gamma=1+(2-\gamma)/N$. Then,  
\[
\begin{aligned}
	\tilde J &  \leq 
	\left\{
	\begin{aligned}
	&C \int_0^{t/2} + C \int_{t/2}^t s^{-\frac{(2-\gamma)(p-\alpha)}{2(p-1)}} 
	\left(\frac{s(t-s)}{t}\right)^{-\frac{\gamma}{2}} ds  &&(p>p_\gamma) \\
	&C \int_0^{t/2} + C \int_{t/2}^t 
	s^{-\frac{N}{2}(p-1)} \left( \log(e+s^{-\frac{1}{2}}) \right)^{-1-\beta} 
	\left(\frac{s(t-s)}{t}\right)^{-\frac{\gamma}{2}} ds  &&(p=p_\gamma)
	\end{aligned}
	\right.  \\
	&\leq 
	\left\{
	\begin{aligned}
	&C \int_0^{t/2} s^{-\frac{(2-\gamma)(p-\alpha)}{2(p-1)} -\frac{\gamma}{2}} ds 
	+ C t^{-\frac{(2-\gamma)(p-\alpha)}{2(p-1)}}  
	\int_{t/2}^t (t-s)^{-\frac{\gamma}{2}} ds  &&(p>p_\gamma) \\
	&C \int_0^{t/2} s^{-1} \left( \log(e+s^{-\frac{1}{2}}) \right)^{-1-\beta} ds
	+ C t^{\frac{\gamma}{2}-1} \left( \log(e+t^{-\frac{1}{2}}) \right)^{-1-\beta} 
	\int_{t/2}^t (t-s)^{-\frac{\gamma}{2}} ds  &&(p=p_\gamma). 
	\end{aligned}
	\right.  \\
\end{aligned}
\]
Hence the same computations as \eqref{eq:J1escal} yield 
\[
	\tilde J(x,t,\eta) \leq 
	\left\{
	\begin{aligned}
	&C t^\frac{(2-\gamma)(\alpha-1)}{2(p-1)}
	&&\mbox{ if }p>p_\gamma, \\
	&C (1+t^\frac{1}{2})  \left( \log(e+t^{-\frac{1}{2}}) \right)^{-\beta} 
	&&\mbox{ if }p=p_\gamma. 
	\end{aligned}
	\right.  
\]
From this, it follows that 
\[
\begin{aligned}
	&\int_0^t \int_{\R^N} G(x-y,t-s) |y|^{-\gamma} \overline{v}(y,s)^p  dyds \\
	&\leq  
	\left\{ 
	\begin{aligned}
	&C c^p t^\frac{(2-\gamma)(\alpha-1)}{2(p-1)}
	\int_{\R^N}  G(x-\eta,t) H(g(\eta))  d\eta 
	&&\mbox{ if }p>p_\gamma, \\
	&C c^p (1+t^\frac{1}{2}) \left( \log(e+ t^{-\frac{1}{2}}) \right)^{-\beta}
	\int_{\R^N}  G(x-\eta,t) H(g(\eta))  d\eta 
	&&\mbox{ if }p=p_\gamma. 
	\end{aligned}
	\right. 
\end{aligned}
\]
By \eqref{eq:Vestiuni} together with 
the same computations as \eqref{eq:GHflin}, we have 
\[
	\int_{\R^N}  G(x-\eta,t) H(g(\eta)) d\eta
	\leq \frac{H(\|V(\cdot,t)\|_\infty)}{\|V(\cdot,t)\|_\infty} V(x,t) 
	\leq 
	\left\{ 
	\begin{aligned}
	&Ct^{-\frac{(2-\gamma)(\alpha-1)}{2(p-1)}} V&&\mbox{ if }p>p_\gamma, \\
	&C\left( \log(e+t^{-\frac{1}{2}}) \right)^\beta V &&\mbox{ if }p>p_\gamma,
	\end{aligned}
	\right.
\]
so that 
\[
	\int_0^t \int_{\R^N} G(x-y,t-s) |y|^{-\gamma} \overline{v}(y,s)^p  dyds
	\leq 
	\left\{ 
	\begin{aligned}
	&Cc^{p-1} \overline{v}&&\mbox{ if }p>p_\gamma, \\
	&Cc^{p-1} (1+t^\frac{1}{2}) \overline{v} &&\mbox{ if }p>p_\gamma. 
	\end{aligned}
	\right.
\]

The above computations imply that 
\[
\begin{aligned}
	\Phi[\overline{v}]
	&\leq 
	\left\{
	\begin{aligned}
	&(2^{-1} + C c^{p-1}) \overline{v} 
	&&\mbox{ if }p>p_\gamma, \\
	&(2^{-1} + C c^{p-1}  (1+t^\frac{1}{2}) ) \overline{v} 
	&&\mbox{ if }p=p_\gamma, 
	\end{aligned}
	\right. 
\end{aligned}
\]
where $C>0$ is a constant depending on $N$, $p$, $\gamma$ 
but not on $c$. 
Then $\overline{v}$ is a supersolution of \eqref{eq:main} 
if we restrict the range of $t$ suitably and $c$ is small. 
The proof is complete. 
\end{proof}

\section{Proof of Theorem \ref{th:origin} for $p< p_\gamma$}\label{sec:1.2sub}
Let $z\in \R^N$ and $\phi\in L^\infty(\R^N)$. For $c>0$, we define
\[
	\tilde w(x,t) := 2 c G(x-z, t) + 2C_0, 
\]
where $C_0 :=\|\phi\|_{L^\infty(\R^N)}$.

\begin{proof}
By Lemma~\ref{lem:super}, it suffices to construct a supersolution of \eqref{eq:main}.
The assumption on $u_0$ gives 
\[
	\int_{\RN} G(x-y,t) d u_0(y) 
	\leq cG(x-z,t) + C_0 =\frac{1}{2} \tilde w(x,t). 
\]
From the same computations as \eqref{eq:Jtileq}, 
$G(y-z,s)\leq Cs^{-N/2}$ and \eqref{eq:Gxi}, it follows that 
\[
\begin{aligned}
	&\int_0^t \int_{\R^N} G(x-y,t-s) |y|^{-\gamma} \tilde w(y,s)^p  dyds \\
	&\leq  C c^p G(x-z,t) 
	\int_0^t \int_{\R^N} 
	G\left(y-\frac{s}{t}x-\frac{t-s}{t}z, \frac{s(t-s)}{t} \right)
	|y|^{-\gamma} G(y-z,s)^{p-1} dyds + C C_0^p t^{1-\frac{\gamma}{2}} \\
	&\leq 
	C c^p G(x-z,t)
	\int_0^t s^{-\frac{N}{2}(p-1)} 
	\left(\frac{s(t-s)}{t}\right)^{-\frac{\gamma}{2}} ds 
	+ C C_0^p t^{1-\frac{\gamma}{2}}. 
\end{aligned}
\]
Since $p<p_\gamma$, we have 
\[
	\int_0^t s^{-\frac{N}{2}(p-1)} 
	\left(\frac{s(t-s)}{t}\right)^{-\frac{\gamma}{2}} ds 
	\leq 
	C\int_0^{t/2} s^{-\frac{N}{2}(p-1)-\frac{\gamma}{2}} ds 
	+ C\int_{t/2}^t (t-s)^{{-\frac{N}{2}(p-1)-\frac{\gamma}{2}}} ds 
	\leq 
	C t^{\frac{N}{2}(p_\gamma-p)}. 
\]
Then, 
\[
	\Phi[\tilde w] \leq \frac{1}{2} \tilde w 
	+ C(c^{p-1}t^{\frac{N}{2}(p_\gamma-p)})c G(x-z,t)
	+CC_0^p t^{1-\frac{\gamma}{2}}, 
\]
where $C>0$ is a constant independent of $c$. 
By restricting the range of $t$ and taking a small constant $c>0$, 
we see that $\tilde w$ is a supersolution of \eqref{eq:main}. 
The proof is complete. 
\end{proof}

\section{Fractional Hardy parabolic equation}\label{sec:frac}
In this section, we consider the Cauchy problem for the fractional Hardy parabolic equation 
\begin{equation}\label{eq:fracmain}
\left\{ 
\begin{aligned}
	&\partial_t u+(-\Delta)^\frac{\theta}{2} u=|x|^{-\gamma}u^p
	&&\mbox{ in }\R^N\times(0,T), \\
	&u(\cdot,0)=u_0 &&\mbox{ in }\R^N, 
\end{aligned}
\right.
\end{equation}
where $N\geq1$, $0<\theta<2$, $p>1$, $0<\gamma<\min\{\theta,N\}$ 
and $0<T\leq \infty$. Set 
\[
	p_{F, \theta}:= 1+\frac{\theta}{N}, \qquad 
	p_{\gamma, \theta}:=1+\frac{\theta-\gamma}{N}
\]
and 
\begin{equation}\label{eq:fracpsizdef}
	\psi_\theta(z):=
	\left\{ 
	\begin{aligned}
	&|z|^\frac{\theta}{p-1}(1+|z|)^{-\frac{\theta-\gamma}{p-1}}
	&&\mbox{ if $p<p_{F, \theta}$ with $N\geq1$ or $p=p_{F, \theta}$ with $N\leq\theta$}, \\
	&|z|^\frac{\gamma}{p-1}  
	&&\mbox{ if $p=p_{F, \theta}$ with $N>\theta$ or $p>p_{F, \theta}$ with $N\geq1$}. 
	\end{aligned}
	\right.
\end{equation}

Our results for \eqref{eq:fracmain} are as follows.

\begin{theorem}\label{th:fracmain}
Fix $N\geq1$, $0<\theta<2$, $p>1$, $0<\gamma<\min\{\theta,N\}$, $0<T<\infty$ and 
$z\in \R^N\setminus\{0\}$. 
Let $\psi_\theta$ be as in \eqref{eq:fracpsizdef}. 
Assume either $u_0$ is a nonnegative Radon measure satisfying 
\[
	u_0=c \psi_\theta(z) \delta_z + \phi\qquad \mbox{ if }p<p_{F,\theta} 
\]
for a nonnegative function $\phi\in L^\infty(\R^N)$ 
with $\|\phi\|_{L^\infty(\R^N)}\leq C_0$, 
or $u_0$ is a nonnegative measurable function satisfying
\[
	u_0(x)\leq \left\{ 
	\begin{aligned}
	& c \psi_\theta(z) |x-z|^{-N} 
	\left( \log( e+|x-z|^{-1}) \right)^{-\frac{N}{\theta}-1}
	\chi_1(|x-z|) + C_0 
	&&\mbox{ if }p=p_{F, \theta}, \\
	& c \psi_\theta(z) |x-z|^{-\frac{\theta}{p-1}} + C_0 
	&&\mbox{ if }p>p_{F, \theta} \\
	\end{aligned}
	\right.
\]
for any $x\in \R^N\setminus\{z\}$. 
Here $c>0$ and $C_0\geq0$ are constants. 
Then there exist positive constants $c_*$ and $C_*$ 
depending on $N$, $\theta$, $p$ and $\gamma$ but not on $T$ and $z$ 
such that the following statements hold. 
If the constants $c$ and $C_0$ satisfy 
\[
	c\leq 
	\left\{
	\begin{aligned} 
	&c_* T^{-\frac{N(p_{F, \theta}-p)}{\theta(p-1)}} 
	(1+T^{1-\frac{\gamma}{\theta}  })^{-\frac{1}{p-1}} 
	&&\mbox{ if }p<p_{F, \theta}, \\
	&c_* (1+T^\frac{1}{\theta}+T^{1-\frac{\gamma}{\theta}})^{-\frac{1}{p-1}} 
	&&\mbox{ if }p=p_{F, \theta}, \\
	&c_* 
	&&\mbox{ if }p>p_{F, \theta}, 
	\end{aligned} 
	\right.
	\qquad 
	C_0 \leq C_* T^{-\frac{\theta-\gamma}{\theta(p-1)}}, 
\]
respectively, then \eqref{eq:fracmain} possesses a solution on $\R^N\times[0,T)$. 
In addition, if $p>p_{F,\theta}$, $c\leq c_*$ and $C_0=0$, 
then \eqref{eq:fracmain} possesses a solution on $\R^N\times[0,\infty)$. 
\end{theorem}

\begin{remark}
We can improve $\psi_\theta$ to $|z|^{\gamma/(p-1)}$ for each case 
if $T$ and $z\neq0$ satisfy
$0<T\leq |z|^\theta$. 
\end{remark}

\begin{theorem}\label{th:fracorigin}
Fix $N\geq1$, $0<\theta<2$, $p>1$, $0<\gamma<\min\{\theta,N\}$, $0<T<\infty$ and 
$z\in \R^N$. 
Assume either $u_0$ is a nonnegative Radon measure satisfying 
\[
	u_0=c \delta_z + \phi \qquad \mbox{ if }p<p_{\gamma, \theta} 
\]
for a nonnegative function $\phi\in L^\infty(\R^N)$ 
with $\|\phi\|_{L^\infty(\R^N)}\leq C_0$, 
or $u_0$ is a nonnegative measurable function satisfying 
\[
	u_0(x)\leq \left\{ 
	\begin{aligned}
	& c |x-z|^{-N} 
	\left( \log ( e+|x-z|^{-1} ) \right)^{-\frac{N}{\theta-\gamma}-1}
	\chi_1(|x-z|) + C_0 
	&&\mbox{ if }p=p_{\gamma, \theta}, \\
	& c |x-z|^{-\frac{\theta-\gamma}{p-1}} + C_0 
	&&\mbox{ if }p>p_{\gamma, \theta}
	\end{aligned}
	\right.
\]
for any $x\in \R^N\setminus\{z\}$. 
Here $c>0$ and $C_0\geq0$ are constants. 
Then there exist positive constants $c_*$ and $C_*$ 
depending on $N$, $\theta$, $p$ and $\gamma$ but not on $T$ and $z$ 
such that the following statements hold. 
If the constants $c$ and $C_0$ satisfy 
\[
	c\leq 
	\left\{
	\begin{aligned} 
	&c_* T^{-\frac{N(p_{F, \theta}-p)}{\theta(p-1)}} &&\mbox{ if }p<p_{\gamma, \theta}, \\
	&c_* (1+T^\frac{1}{\theta})^{-\frac{1}{p-1}} &&\mbox{ if }p=p_{\gamma, \theta}, \\
	&c_* 
	&&\mbox{ if }p>p_{\gamma, \theta}, 
	\end{aligned} 
	\right.
	\qquad 
	C_0 \leq C_* T^{-\frac{\theta-\gamma}{\theta(p-1)}}, 
\]
respectively, then \eqref{eq:fracmain} possesses a solution on $\R^N\times[0,T)$. 
In addition, if $p>p_{\gamma, \theta}$, $c\leq c_*$ and $C_0=0$, 
then \eqref{eq:fracmain} possesses a solution on $\R^N\times[0,\infty)$. 
\end{theorem}

\begin{remark}
We remark that each of the singularities in 
Theorems {\rm \ref{th:fracmain}} and {\rm \ref{th:fracorigin}} 
is optimal at least we additionally assume that 
$\theta$, $p$ and $\gamma$ satisfy $0<\gamma<\theta(p-1)$. 
This assumption is needed to show necessary conditions 
for the existence of solutions of \eqref{eq:fracmain}. 
For more details of necessary conditions, see {\rm \cite{HS}}. 
\end{remark}

The key to prove the above theorems is the following Lemma. 

\begin{lemma}[{\cite[Inequality (9)]{BJ}}]
Let $G_\theta$ be the fundamental solution of 
the fractional heat equation
$\partial_t u + (-\Delta)^{\theta/2} u = 0$ in $\RN\times(0,\infty)$. 
Then there exists $C>0$ depending only on $N$ and $\theta$ such that, 
for any $x,y,\eta\in\RN$ and $0<s<t$, 
\[
G_\theta(x-y,t-s)G_\theta(y-\eta,s) \leq CG_\theta(x-\eta,t)(G_\theta(y-\eta,s) + G_\theta(y-x,t-s)). 
\]
\end{lemma}

By applying this lemma instead of Lemmas \ref{lem:Gtra} and \ref{lem:Gmid} 
and using upper and lower estimates of $G_\theta$ (see for instance \cite{BGS, Su}), 
we can prove Theorems \ref{th:fracmain} and \ref{th:fracorigin} 
in much the same way as Theorems \ref{th:main} and \ref{th:origin}. 
We leave the details to the reader.

\section*{Acknowledgment} 
The first author was supported in part by JSPS KAKENHI Grant Number JP19H05599.
The second author was supported in part by JSPS KAKENHI Grant Number 19K14567.





\end{document}